\documentclass[11pt]{article}

\date{}
\title{Infinitely many minimal classes of graphs \\ of unbounded clique-width\thanks{The authors acknowledge support of EPSRC, grant EP/L020408/1}}
\author{A. Collins,
  J. Foniok\thanks{School of Computing, Mathematics and Digital Technology, Manchester
  Metropolitan University, John Dalton Building, Chester Street, Manchester, M1 5GD, United Kingdom},
N. Korpelainen, V. Lozin, V. Zamaraev}

\usepackage[cp1251]{inputenc}
\usepackage{amsmath, amsthm, amssymb}
\usepackage{colordvi}
\usepackage{graphicx}
\usepackage{tikz}
\usetikzlibrary{decorations.pathreplacing,arrows,automata}

\tikzstyle{vertex}=[circle,fill=black!100,text=white,inner sep=0.8mm]
\tikzstyle{point}=[circle,fill=black,inner sep=0.1mm]

\begin{document}
\maketitle

\newtheorem{theorem}{Theorem}
\newtheorem{lemma}{Lemma}
\newtheorem{proposition}{Proposition}
\newtheorem{prop}{Proposition}
\newtheorem{cor}{Corollary}
\newtheorem{definition}{Definition}
\newtheorem{remark}{Remark}
\newtheorem{conjecture}{Conjecture}

\def\N{\mathbb{N}}
\def\t{\sim}
\def\nt{\nsim}
\def\1{n+1}
\def\2{n+2}

\begin{abstract}
The celebrated theorem of Robertson and Seymour states that in the family of minor-closed 
graph classes, there is a unique minimal class of graphs of unbounded tree-width, namely, 
the class of planar graphs. In the case of tree-width, the restriction to minor-closed
classes is justified by the fact that the tree-width of a graph is never smaller than 
the tree-width of any of its minors. This, however, is not the case with respect to clique-width,
as the clique-width of a graph can be (much) smaller than the clique-width of its minor.
On the other hand, the clique-width of a graph is never smaller than the clique-width of 
any of its induced subgraphs, which allows us to be restricted to hereditary classes
(that is, classes closed under taking induced subgraphs), when we study clique-width. 
Up to date, only finitely many minimal hereditary classes of graphs of unbounded clique-width
have been discovered in the literature. In the present paper, we prove that the family of such 
classes is infinite. Moreover, we show that the same is true with respect to linear clique-width.    

\end{abstract}




\section{Introduction}

Clique-width is a graph parameter which is important in theoretical computer science, because many 
algorithmic problems that are generally NP-hard become polynomial-time solvable when restricted 
to graphs of bounded clique-width \cite{CMR00}. Clique-width is a relatively new notion and it generalises another 
important graph parameter, tree-width, studied in the literature for decades. Clique-width is 
stronger than tree-width in the sense that graphs of bounded tree-width have bounded clique-width,
but not necessarily vice versa. For instance, both parameters are bounded for trees, while
for complete graphs only clique-width is bounded. 

When we study classes of graphs of bounded tree-width, we may assume without loss of generality that  
together with every graph $G$ our class contains all minors of $G$, as the tree-width of a minor can 
never be larger than the tree-width of the graph itself. In other words, when we try to identify classes 
of graphs of bounded tree-width, we may restrict ourselves to minor-closed graph classes. However, when 
we deal with clique-width this restriction is not justified, as the clique-width of a minor of $G$ can be
much larger than the clique-width of $G$. On the other hand, the clique-width of $G$ is never smaller 
than the clique-width of any of its induced subgraphs \cite{CO00}. This allows us to be restricted to hereditary classes,
that is, those that are closed under taking induced subgraphs. 

One of the most remarkable outcomes of the graph minor project of Robertson and Seymour is the proof of Wagner's conjecture
stating that the minor relation is a well-quasi-order \cite{RS88}. This implies, in particular, that in the world of 
minor-closed graph classes there exist minimal classes of unbounded tree-width and the number of such classes is finite. 
In fact, there is just one such class (the planar graphs), which was shown even before the proof of Wagner's conjecture \cite{RS86}. 

In the world of hereditary classes the situation is more complicated, because the induced subgraph relation is not a well-quasi-order. 
It contains infinite antichains, and hence, there may exist infinite strictly decreasing sequences of graph classes with no minimal one. 
In other words, even the existence of minimal hereditary classes of unbounded clique-width is not an obvious fact. This fact was 
recently confirmed in \cite{Lozin}. However, whether the number of such classes is finite or infinite
remained an open question. In the present paper, we settle this question by showing that the family of minimal hereditary classes 
of unbounded clique-width is infinite. Moreover, we prove that the same is true with respect to {\it linear} clique-width. 

The organisation of the paper is as follows. In the next section, we introduce basic notation and terminology. 
In Section~\ref{sec:minimal}, we describe a family of graph classes of unbounded clique-width and prove that 
infinitely many of them are minimal with respect to this property. In Section~\ref{sec:more}, we identify more
classes of unbounded clique-width. Finally, Section~\ref{sec:conclusion} concludes the paper with a number of open problems. 

\section{Preliminaries}
\label{sec:pre}
All graphs in this paper are undirected, without loops and multiple edges. 
For a graph $G$, we denote by $V(G)$ and $E(G)$ the vertex set and the edge set of $G$, respectively. 
The \emph{neighbourhood} of a vertex $v\in V(G)$ is the set of vertices adjacent to $v$ and the \emph{degree} of $v$
is the size of its neighbourhood. As usual, by $P_n$ and $C_{n}$ we denote a chordless path and a chordless cycle with $n$ vertices, respectively.

In a graph, an {\it independent set} is a subset of vertices no two of which are adjacent. A graph is
\emph{bipartite} if its vertices can be partitioned 
into two independent sets. Given a bipartite graph $G$ together with a bipartition of its vertices into two independent sets $V_1$ and $V_2$,
the \emph{bipartite complement} of $G$ is the bipartite graph obtained from $G$ by complementing the edges {\it between} $V_1$ and $V_2$.

Let $G$ be a graph and $U\subseteq V(G)$ a subset of its vertices. Two vertices of $U$ will be called
\emph{$U$-similar} if they have the same neighbourhood outside $U$.  
Clearly, $U$-similarity is an equivalence relation.  The number of equivalence classes of $U$-similarity will be denoted $\mu(U)$. 
Also, by $G[U]$ we will denote the subgraph of $G$ 
induced by $U$, that is, the subgraph of $G$ with vertex set $U$ and two vertices being
adjacent in $G[U]$ if and only if they are adjacent in $G$. 
We say that a graph $H$ is an \emph{induced subgraph} of $G$ if $H$ is isomorphic to $G[U]$ for some $U\subseteq V(G)$. 

A class $X$ of graphs is {\it hereditary} if it is closed under taking induced subgraphs, that is, $G\in X$ implies $H\in X$
for every induced subgraph $H$ of $G$.
It is well-known that a class of graphs is hereditary if and only if it can be characterised in terms of 
forbidden induced subgraphs. More formally, given a set of graphs $M$, we say that 
a graph $G$ is \emph{$M$-free} if $G$ does not contain induced subgraphs isomorphic to graphs in $M$.
Then a class $X$ is hereditary if and only if graphs in $X$ are $M$-free for a set $M$.

\medskip
The notion of clique-width of a graph was introduced in \cite{CER93}. The clique-width of a graph $G$ is denoted ${\rm cwd}(G)$
and is defined as the minimum number of labels needed to construct $G$ by means of the following four graph operations: 
\begin{itemize}
\item
creation of a new vertex $v$ with 
label $i$ (denoted $i(v)$), 
\item disjoint union of two labelled graphs $G$ and 
$H$ (denoted $G\oplus H$), 
\item connecting vertices with specified labels $i$ 
and $j$ (denoted $\eta_{i,j}$) and 
\item renaming label $i$ to label $j$ 
(denoted $\rho_{i\to j}$). 
\end{itemize}

Every graph can be defined by an algebraic expression using the four operations above. 
This expression is called a $k$-expression if it uses $k$ different labels.
For instance, the cycle $C_5$ on vertices $a,b,c,d,e$ 
(listed along the cycle) can be defined by the following 4-expression:
$$
\eta_{4,1}(\eta_{4,3}(4(e)\oplus\rho_{4\to 3}(\rho_{3\to 2}(\eta_{4,3}(4(d)\oplus\eta_{3,2}(3(c)\oplus\eta_{2,1}(2(b)\oplus 1(a)))))))).
$$

Alternatively, any algebraic expression defining $G$ can be represented as a rooted tree, 
whose leaves correspond to the operations of vertex creation, the internal nodes correspond 
to the $\oplus$-operations, and the root is associated with $G$. The operations $\eta$ and 
$\rho$ are assigned to the respective edges of the tree. Figure~\ref{fig:tree} shows the tree 
representing the above expression defining a $C_5$. 

\bigskip
\begin{figure}[ht]
\begin{center}
\setlength{\unitlength}{0.32mm}
\begin{picture}(370,50)
\put(15,50){\circle{20}}
\put(90,50){\circle{20}}
\put(165,50){\circle{20}}
\put(240,50){\circle{20}}
\put(315,50){\circle{20}}
\put(355,50){\circle{20}}
\put(90,10){\circle{20}}
\put(165,10){\circle{20}}
\put(240,10){\circle{20}}
\put(315,10){\circle{20}}
\put(25,50){\line(1,0){55}}
\put(100,50){\line(1,0){55}}
\put(175,50){\line(1,0){55}}
\put(250,50){\line(1,0){55}}
\put(325,50){\line(1,0){20}}
\put(90,40){\line(0,-1){20}}
\put(165,40){\line(0,-1){20}}
\put(240,40){\line(0,-1){20}}
\put(315,40){\line(0,-1){20}}
\put(85,47){+}
\put(160,47){+}
\put(235,47){+}
\put(311,47){+}
\put(8,47){$C_5$}
\put(83,10){$_{4(e)}$}
\put(158,10){$_{4(d)}$}
\put(233,10){$_{3(c)}$}
\put(308,10){$_{2(b)}$}
\put(348,50){$_{1(a)}$}
\put(101,55){$_{\rho_{4\to 3}\rho_{3\to 2}\eta_{4,3}}$}
\put(36,55){$_{\eta_{4,1}\eta_{4,3}}$}
\put(196,55){$_{\eta_{3,2}}$}
\put(271,55){$_{\eta_{2,1}}$}
\end{picture}
\end{center}
\caption{The tree representing the expression defining a $C_5$}
\label{fig:tree}
\end{figure}
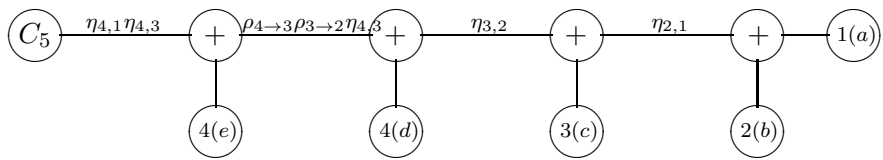

Let us observe that the tree in Figure~\ref{fig:tree} has a special form known as a {\it caterpillar
tree} (that is, a tree that becomes a path 
after the removal of vertices of degree 1). The minimum number of labels needed to construct a graph $G$ by means of caterpillar trees is 
called the {\it linear clique-width} of $G$ and is denoted ${\rm lcwd}(G)$. Clearly, ${\rm lcwd}(G) \ge {\rm cwd}(G)$ and there are classes 
of graphs for which the difference between clique-width and linear clique-width can be arbitrarily large (see e.g. \cite{Robert}).

\medskip 
A notion which is closely related to clique-width is that of {\it rank-width} (denoted $\text{rwd}(G)$), which was introduced by Oum and Seymour in \cite {oum}. 
They showed that rank-width and clique-width are related to each other by proving that if the clique-width of a graph $G$ is $k$, 
then $$\text{rwd}(G)\leq k\leq 2^{\text{rwd}(G)+1}-1.$$ Therefore a class of graphs has unbounded clique-width if and only if it also has unbounded rank-width. 

For a graph $G$ and a vertex $v$, the ${\it local}$ ${\it complementation}$ at $v$ is the operation that 
replaces the subgraph induced by the neighbourhood of $v$ with its complement. 
A graph $H$ is a ${\it vertex}$-${\it minor}$ of $G$ if $H$ can be obtained from $G$ by a sequence 
of local complementations and vertex deletions. In \cite {oum2} it was proved that if $H$ is a vertex-minor of $G$, 
then the rank-width of $H$ is at most the rank-width of $G$. 

\medskip
Finally, we introduce some language-theoretic terminology and notation. 
Given a word $\alpha$, we denote by $\alpha(k)$ the $k$-th letter of $\alpha$ and by $\alpha^k$ the concatenation of $k$ copies of $\alpha$.
A {\it factor} of $\alpha$ is a contiguous subword of $\alpha$, that is, a subword $\alpha(i)\alpha(i+1)\ldots\alpha(i+k)$ for some $i$ and $k$.
An infinite word $\alpha$ is {\it periodic} if there is a positive integer~$k$ such that $\alpha(i)=\alpha(i+k)$ for all $i$.


\section{Minimal classes of graphs of unbounded clique-width}
\label{sec:minimal}

In this section, we describe an infinite family of graph classes of unbounded clique-width (Subsections~\ref{sec:basic} and~\ref{sec:other}). 
The fact that each of them is a minimal hereditary class of unbounded clique-width will be proved in Subsection~\ref{sec:min}. 

Each class in our family is defined through a universal element, that is, an infinite graph that contains all graphs from the class as induced subgraphs.
All constructions start from the graph~$\mathcal{P}$ given by
\begin{align*}
V(\mathcal{P}) &= \{ v_{i,j} : i,j \in \mathbb{N} \}, \\
E(\mathcal{P}) &= \bigl\{ \{ v_{i,j}, v_{i,j+1} \} : i,j \in \mathbb{N} \bigr\}.
\end{align*}
The \emph{$j$th column} of~$\mathcal{P}$ is the set $V_j = \{ v_{i,j} : i \in \mathbb{N} \}$,
and the \emph{$i$th row} of~$\mathcal{P}$ is the set $R_i = \{ v_{i,j} : j \in \mathbb{N} \}$.
Observe that each row of~$\mathcal{P}$ induces an infinite chordless path,
and the graph~$\mathcal{P}$ is the disjoint union of these paths.
Moreover, any two consecutive columns $V_j$ and~$V_{j+1}$ induce a 1-regular graph, that is,
a collection of disjoint edges (one edge from each path). 

Let $\alpha=\alpha_1\alpha_2\ldots$ be an infinite binary word, that is, an infinite word such that $\alpha_j\in \{0,1\}$ for each natural $j$. 
The graph ${\cal P}^{\alpha}$ is obtained from $\cal P$ by complementing the edges between two consecutive columns $V_j$ and $V_{j+1}$ if and only if $\alpha_j=1$.
In other words, we apply bipartite complementation to the bipartite graph induced by $V_j$ and $V_{j+1}$.
In particular, if $\alpha$ does not contain $1$s, then ${\cal P}^{\alpha}=\cal P$. 

Finally, by $\cal G_{\alpha}$ we denote the class of all finite induced subgraphs of ${\cal P}^{\alpha}$.
By definition, $\cal G_{\alpha}$ is a hereditary class. In what follows we show that 
$\cal G_{\alpha}$ is a minimal hereditary class of unbounded clique-width for infinitely many values of $\alpha$.

\subsection{The basic class}
\label{sec:basic}

Our first example constitutes the basis for infinitely many other constructions. It deals with the class ${\cal G}_{1^{\infty}}$, where $1^{\infty}$
stands for the infinite word of all $1$s. Let us denote by 
\begin{itemize}
\item[$F_{n,n}$] the subgraph of ${\cal P}^{1^{\infty}}$ induced by $n$ consecutive columns and any $n$ rows.
\end{itemize}
In order to show that ${\cal G}_{1^{\infty}}$ is a class of unbounded clique-width, 
we will prove the following lemma.

\begin{lemma}\label{cm}
The clique-width of $F_{n,n}$ is at least $\lfloor n/2\rfloor$.
\end{lemma}

\begin{proof}
Let ${\rm cwd}(F_{n,n})=t$. Denote by $\tau$ a $t$-expression defining $F_{n,n}$ and
by ${\it tree}(\tau)$ the rooted tree representing $\tau$.
The subtree of ${\it tree}(\tau)$ rooted at a node $x$ will be denoted
${\it tree}(x,\tau)$. This subtree corresponds to a subgraph of $F_{n,n}$,
which will be denoted $F(x)$. The label of a vertex $v$ of the graph $F_{n,n}$
at the node $x$ is defined as the label that $v$ has immediately prior to applying
the operation $x$.

Let $a$ be a lowest $\oplus$-node in ${\it tree}(\tau)$ such that $F(a)$ contains a full column of $F_{n,n}$.
Denote the children of $a$ in ${\it tree}(\tau)$ by $b$ and $c$. Let us  colour all vertices in
$F(b)$ blue and all vertices in $F(c)$ red, and the remaining vertices of $F_{n,n}$ yellow. 
Note that by the choice of $a$ the graph $F_{n,n}$ contains a non-yellow column (that is, a column each vertex 
of which is non-yellow), but none of its columns are entirely red or blue. 
Let $V_r$ be a non-yellow column of $F_{n,n}$. Without loss of generality we assume
that $r\le \lceil n/2 \rceil$ and that the column $r$ contains at least $n/2$ red vertices,
since otherwise we could consider the columns in reverse order and swap the colours red and blue.

Observe that edges of $F_{n,n}$ between different coloured vertices are not present in $F(a)$. 
Therefore, if a non-red vertex distinguishes two red vertices $u$ and $v$, then $u$ and $v$
must have different labels at the node $a$. We will use this fact to show that $F(a)$ contains 
a set $U$ of at least $\lfloor n/2\rfloor$ vertices with pairwise different labels at the node $a$. 
Such a set can be constructed by the following procedure.

\begin{enumerate} 
\item Set $j = r,\ U=\emptyset$ and $I=\{i : v_{i,r} \mbox{ is
    red}\}$.
\item Set $K= \{i\in I : v_{i,j+1} \mbox{ is non-red}\}$.
\item If $K\not= \emptyset$, add the vertices $\{v_{k,j}: k\in K\}$ to
  $U$. Remove members of $K$ from $I$.
\item If $I=\emptyset$, terminate the procedure.
\item Increase $j$ by 1. If $j=n$, choose an arbitrary $i\in I$, put
  $U=\{v_{i,m}: r\leq m \leq n-1\}$ and terminate the procedure.
\item Go back to Step 2.
\end{enumerate}

It is not difficult to see that this procedure must terminate. To
complete the proof, it suffices to show that whenever the procedure
terminates, the size of $U$ is at least $\lfloor n/2\rfloor$ and 
the vertices in $U$ have pairwise different labels at the node $a$

First, suppose that the procedure terminates in Step 5. 
Then $U$ is a subset of red vertices from at least 
$\lfloor n/2\rfloor$ consecutive columns of row $i$. Consider two vertices 
$v_{i,l}\ , v_{i,m}\in U$ with $l<m$. According to the above procedure, 
$v_{i,m+1,}$ is red. Since $F_{n,n}$ does not contain an entirely red column,
there must exist a non-red vertex $w$ in the column $m+1$. According to the structure
of $F_{n,n}$, vertex $w$ is adjacent to $v_{i,m}$ and non-adjacent to  $v_{i,l}$. 
We conclude that $v_{i,l}$ and $v_{i,m}$ have different labels. Since $v_{i,l}$ and $v_{i,m}$
have been chosen arbitrarily, the vertices of $U$ have pairwise different labels.

Now suppose that the procedure terminates in Step 4. 
By analysing Steps 2 and 3, it is easy to deduce that $U$ is
a subset of red vertices of size at least $\lfloor n/2\rfloor$.
Suppose that $v_{i,l}$ and $v_{k, m}$ are two vertices in $U$ with $l\leq m$. 
The procedure certainly guarantees that $i\not= k$ and that both $v_{i,l+1}$ and 
$v_{k,m+1}$ are non-red. If $m\in \{l, l+2\}$, then it is clear that $v_{i,l+1}$ 
distinguishes vertices $v_{i,l}$ and $v_{k,m}$, and therefore these vertices have different labels.
If $m\notin \{l,l+2\}$, we may consider vertex $v_{k,m-1}$ which must be red. 
Since $F_{n,n}$ does not contain an entirely red column, the vertex $v_{k,m}$ must have a
non-red neighbour $w$ in the column $m-1$. But $w$ is not a neighbour of
$v_{i,l}$, trivially. We conclude that $v_{i,l}$ and $v_{k,m}$ have
different labels, and therefore, the vertices of $U$ have pairwise different labels.  
This shows that the clique-width of the graph $F_{n,n}$ is at least $\lfloor n/2\rfloor$.  
\end{proof}

\subsection{Other classes}
\label{sec:other}

In this section, we discover more hereditary classes of graphs of unbounded clique-width 
by showing that for all $n\in \mathbb{N}$ such classes have graphs containing $F_{n,n}$ 
as a vertex-minor.

\begin{lemma}\label{ucw}
Let $\alpha$ be an infinite binary word containing infinitely many $1$s.
Then the clique-width of graphs in the class ${\cal G}_{\alpha}$ is unbounded.
\end{lemma}

\begin{proof}
First fix an even number $n$. Let $\beta$ be a factor of $\alpha$ containing precisely $n$ occurrences of $1$, 
starting and ending with $1$. We denote the length of $\beta$ by $\ell$ and consider the subgraph $G_n$ of $P^{\alpha}$ induced 
by $\ell+1$ consecutive columns corresponding to $\beta$ and by any $n$ rows.  We will now show that $G_n$ contains the graph $F_{n,n}$ 
defined in Lemma~\ref{cm} as a vertex-minor.

If $\beta$ contains $00$ as a factor, then there are three columns $V_i, V_{i+1}, V_{i+2}$ such that each of $V_i\cup V_{i+1}$ and $V_{i+1}\cup V_{i+2}$ 
induces a 1-regular graph. We apply a local complementation to each vertex of $G_n$ in column $V_{i+1}$ and then delete the vertices of $V_{i+1}$ from
$G_n$. Under this operation, our graph transforms into a new graph where column $V_{i+1}$ is absent, while columns $V_i$ and $V_{i+2}$ induce a 1-regular graph. 
In terms of words, this operation is equivalent to removing one $0$ from the factor $00$. Applying this transformation repeatedly, we can reduce
$G_n$ to an instance corresponding to a word $\beta$ with no two consecutive $0$s.

Now assume $\beta$ contains $01$ as a factor, and let $V_j, V_{j+1}$ and $V_{j+2}$ be three consecutive columns such that 
$V_{j}\cup V_{j+1}$ induces a 1-regular graph, while the edges between $V_{j+1}$ and $V_{j+2}$ form the bipartite complement of a 1-regular graph.
We apply a local complementation to each vertex of $V_{j+1}$ in turn and then delete the vertices of $V_{j+1}$ from $G_n$. 
It is not difficult to see that in the transformed graph the edges between $V_{j}$ and $V_{j+2}$ form the bipartite complement of a matching. 
Looking at the vertices in $V_{j+2}$ we see that for any two vertices $x$ and $y$ in this column, 
when a local complementation is applied at $z\in V_{j+1}$ the adjacency between $x$ and $y$ is complemented if and only if both $x$ and $y$ are adjacent to $z$. 
Since $|V_{j+2}|=n$ is even, we conclude that after $n$ applications of local complementation $V_{j+2}$ remains an independent set.
In terms of words, this operation is equivalent to removing $0$ from the factor $01$. Applying this transformation repeatedly, we can reduce
$G_n$ to an instance corresponding to a word $\beta$ which is free of $0$s.

The above discussion shows that $G_n$ can be transformed by a sequence of local complementations and vertex deletions into $F_{n,n}$.
Therefore, $G_n$ contains the graph $F_{n,n}$ as a vertex-minor. Since $n$ can be arbitrarily large, we conclude that the rank-width,
and hence the clique-width, of graphs in ${\cal G}_{\alpha}$ is unbounded.
\end{proof}


\subsection{Minimality of classes ${\cal G}_{\alpha}$ with a periodic $\alpha$}
\label{sec:min}

In the previous section, we proved that any class ${\cal G}_{\alpha}$ with infinitely many $1$s in $\alpha$ has unbounded clique-width. 
In the present section, we will show that if $\alpha$ is periodic, then ${\cal G}_{\alpha}$ is a minimal hereditary class of graphs of unbounded 
clique-width, provided that $\alpha$ contains at least one $1$. In other words, we will show that in any proper hereditary subclass of ${\cal G}_{\alpha}$
the clique-width is bounded. Moreover, we will show that proper hereditary subclasses of ${\cal G}_{\alpha}$ have bounded {\it linear} clique-width.
To this end, we first prove a technical lemma, which strengthens a similar result given in \cite{Lozin} from clique-width to linear clique-width.
Let us repeat that by $\mu (U)$ we denote the number of similarity classes with respect to an equivalence relation defined in Section~\ref{sec:pre}.

\begin{lemma}\label{lem:min-3}
Let $m\geq 2$ and $\ell$ be positive integers.  Suppose that the vertex set
of $G$ can be partitioned into sets $U_1,U_2,\ldots$ where for each $i$, 
\begin{itemize}
\item[{\rm (1)}] ${\rm lcwd}(G[U_i])\leq m$, 
\item[{\rm (2)}] $\mu(U_i)\leq \ell$ and $\mu(U_1\cup\cdots\cup U_i)\leq\ell$. 
\end{itemize}
Then ${\rm lcwd}(G)\leq \ell (m+1)$.
\end{lemma}

\begin{proof}
If $G[U_1]$ can be constructed with at most $m$ labels and $\mu(U_1)\leq \ell$, then $G[U_1]$ can be constructed with at most $m\ell$ different labels 
in such a way that in the process of construction any two vertices in different equivalence classes of $U_1$ have different labels, 
and by the end of the process any two vertices in the same equivalence class of $U_1$ have the same label. 
In other words, we build $G[U_1]$ with at most $m\ell$ labels and finish the process with at most $\ell$ labels corresponding to the equivalence classes of $U_1$.

Now assume we have constructed the graph $G_i=G[U_1\cup\cdots\cup U_i]$ using $m\ell$ different labels making sure that the construction finishes 
with a set $A$ of at most $\ell$ different labels corresponding to the equivalence classes of $U_1\cup\cdots\cup U_i$. 
By assumption, it is possible to construct $G[U_{i+1}]$ using a set $B$ of at most $m\ell$ different labels such that we finish the process 
with at most $\ell$ labels corresponding to the equivalence classes of $U_{i+1}$. We choose labels so that $A$ and $B$ are disjoint. 
As we construct $G[U_{i+1}]$ join each vertex to its neighbours in $G_i$ to build the graph $G_{i+1}=G[U_1\cup\cdots\cup U_i \cup U_{i+1}]$. 
Notice that any two vertices in the same equivalence class of $U_1\cup\cdots\cup U_i$ or $U_{i+1}$ belong to the same equivalence class 
of $U_1\cup\cdots\cup U_i \cup U_{i+1}$. Therefore, the construction of $G_{i+1}$ can be completed with a set of at most $\ell$ different labels 
corresponding to the equivalence classes of the graph. The conclusion now follows by induction.
\end{proof}

Now let $\alpha$ be an infinite binary periodic word of period $p$ with at least one $1$. 
In the following three lemmas, let $H_{k,t}$ be any subgraph of~$\mathcal{P}^\alpha$ induced by the first
$k$~rows and any $t$~consecutive columns.

It is not difficult to see the following fact.
\begin{lemma}\label{lem:univer}
A graph with $n$ vertices in ${\cal G}_{\alpha}$ is an induced subgraph of $H_{k,t}$ for any $k\ge n$ and any $t\ge n(p+1)$.
\end{lemma}

Now, with the help of Lemma~\ref{lem:min-3} we derive the following conclusion. 

\begin{lemma}\label{easy}
The linear clique-width of $H_{k,t}$ is at most $4t$.
\end{lemma}
\begin{proof} 
Denote by $U_i$ the $i$-th row of $H_{k,t}$.  
Since each row induces a path forest (that is, a disjoint union of paths), it is clear that ${\rm lcwd}(G[U_i])\leq 3$ for every $i$. 
Trivially, $\mu (U_i)\leq t$, since $|U_i|=t$. Also, denoting $W_i:=U_1\cup \ldots \cup U_i$, 
it is not difficult to see that $\mu (W_i)\leq t$ for every $i$, since the vertices of the same column are $W_i$-similar. 
Now the conclusion follows from Lemma~\ref{lem:min-3}.
\end{proof}

Next we use Lemmas~\ref{lem:min-3},~\ref{lem:univer}  and~\ref{easy} to prove the following result. 

\begin{lemma}\label{lem:kFk}
For any fixed $k\ge 1$, the linear clique-width of any $H_{k,k}$-free graph~$G$ in the
class~$\mathcal{G}_\alpha$ is at most $(4k-2)(8k+1)$.
\end{lemma}

\begin{proof}
Let $G$ be an $H_{k,k}$-free graph in ${\cal G}_{\alpha}$. By Lemma~\ref{lem:univer}, the graph $G$ is an induced subgraph of $H_{n,n}$ for some $n$. 
For convenience, assume that $n$ is a multiple of $k$, say $n=tk$. We fix an arbitrary embedding of $G$ into $H_{n,n}$ and call
the vertices of $H_{n,n}$ that induce $G$ {\it black}. The remaining vertices of $H_{n,n}$ will be called {\it white}.

For $1\leq i \leq t$, let us denote by $W_i$ the subgraph of $H_{n,n}$ induced by 
the $k$ consecutive columns $(i-1)k+1,(i-1)k+2,\ldots,ik$. 
We partition the vertices of $G$ into subsets $U_1, U_2,\ldots, U_t$ according to 
the following procedure: 

\begin{enumerate}
\item For $1\leq j\leq t$, set $U_j = \emptyset$. Add every black vertex of $W_1$ to $U_1$. Set $i=2$.
\item For $j=1,\ldots,n$, 
\begin{itemize}
\item if row $j$ of $W_i$ is entirely black, then add the first vertex of this row  
to $U_{i-1}$ and the remaining vertices of the row to $U_i$.
\item otherwise, add the (black) vertices of row $j$ preceding the first white vertex
to $U_{i-1}$ and add the remaining black vertices of the row to $U_i$.
\end{itemize}
\item Increase $i$ by 1. If $i=t+1$, terminate the procedure.
\item Go back to Step 2.
\end{enumerate}

Let us show that the partition $U_1, U_2,\ldots, U_t$ given by the procedure satisfies 
the assumptions of Lemma~\ref{lem:min-3} with $m$ and $\ell$ depending only on $k$.

The procedure clearly assures that each $G[U_i]$ is an induced subgraph of 
$G[V(W_i) \cup V(W_{i+1})]$. 
By Lemma~\ref{easy}, we have ${\rm lcwd}(G[V(W_i) \cup V(W_{i+1})]) = {\rm lcwd}(F_{n,2k})\leq 8k$. 
Since the linear clique-width of an induced subgraph cannot exceed the linear clique-width of 
the parent graph, we conclude that ${\rm lcwd}(G[U_j])\leq 8k$, which shows condition 
(1) of Lemma~\ref{lem:min-3}.

To show condition (2) of Lemma~\ref{lem:min-3}, let us call a vertex $v_{j,m}$ of $U_i$ 
boundary if either $v_{j,m-1}$ belongs to $U_{i-1}$ or $v_{j,m+1}$ belongs to $U_{i+1}$
(or both). It is not difficult to see that a vertex of $U_i$ is boundary if it belongs
either to the second column of an entirely black row of $W_i$ or to the first column
of an entirely black row of $W_{i+1}$. Since the graph $G$ is $H_{k,k}$-free, 
the number of rows of $W_i$ which are entirely black is at most $k-1$. 
Therefore, the boundary vertices of $U_i$ introduce at most $2(k-1)$ equivalence classes in $U_i$.  

Now consider two non-boundary vertices of $U_i$ from
the same column. It is not difficult to see that these vertices have the same neighbourhood 
outside of $U_i$. Therefore, the non-boundary vertices of the same column of $U_i$ are $U_i$-similar 
and hence the non-boundary vertices give rise to at most $2k$
equivalence classes in $U_i$. Thus, $\mu(U_i)\leq 4k-2$ for all $i$. 

Similar argument show that $\mu (U_1\cup \ldots\cup U_i)\leq 3k-1 \leq 4k-2$ for all $i$. 
Therefore, by Lemma~\ref{lem:min-3}, we conclude that ${\rm lcwd}(G)\leq (4k-2)(8k+1)$, which completes the proof.
\end{proof}

\begin{theorem}\label{thm:min}
Let $\alpha$ be an infinite binary periodic word containing at least one $1$. 
Then the class ${\cal G}_{\alpha}$ is a minimal hereditary class of graphs of unbounded clique-width and linear clique-width. 
\end{theorem}

\begin{proof}
By Lemma~\ref{ucw}, the clique-with of graphs in ${\cal G}_{\alpha}$ is unbounded. Therefore, linear clique-width is unbounded too.
To prove the minimality, consider a proper hereditary subclass $X$ of ${\cal G}_{\alpha}$ and let $G\in
{\cal G}_{\alpha} \setminus X$.
By Lemma~\ref{lem:univer}, $G$ is an induced subgraph of $H_{k,k}$ for some finite $k$. Therefore, each graph in $X$ is $H_{k,k}$-free.
Observe that the value of $k$ is the same for all graphs in $X$. It depends only on $G$ and the period of $\alpha$.
Therefore, by Lemma~\ref{lem:kFk}, the linear clique-width (and hence clique-width) of graphs in $X$ is bounded by a constant.
\end{proof}


\section{More classes of graphs of unbounded clique-width}
\label{sec:more}

In this section, we extend the alphabet from $\{0,1\}$ to $\{0,1,2\}$ in order to construct more classes of graphs of unbounded clique-width.
Let $\alpha$ be an infinite word over the alphabet $\{0,1,2\}$. We remind the reader that the letter $1$ stands for the operation of bipartite 
complementation between two consecutive columns $V_j$ and $V_{j+1}$ of the graph ${\cal P}$, that is, if $\alpha_{j}=1$, 
then two vertices $v_{i,j}\in V_j$ and $v_{k,j+1}\in V_{j+1}$ are adjacent in ${\cal P}^{\alpha}$ if and only if they are not adjacent in ${\cal P}$.

The new letter $2$ will represent the operation of ``forward'' complementation, that is, if $\alpha_{j}=2$, 
then two vertices $v_{i,j}\in V_j$ and $v_{k,j+1}\in V_{j+1}$ with $i<k$ are adjacent in ${\cal P}^{\alpha}$ if and only if they are not adjacent in ${\cal P}$. 
In other words, this operation adds edges between $v_{i,j}$ and $v_{k,j+1}$ with $i<k$. The bipartite graph induced by two consecutive columns 
corresponding to the letter 2 is known in the literature as a {\it chain graph}.

Of special interest for the topic of this paper is the word $2^{\infty}=222\ldots$.  
The class ${\cal G}_{2^{\infty}}$ is also known as the class of bipartite permutation graphs
and this is one of the first two minimal classes of graphs of unbounded clique-width discovered in the literature \cite{Lozin}.
We will denote by 
\begin{itemize}
\item[$X_{n,n}$] the subgraph of ${\cal P}^{2^{\infty}}$ induced by $n$ consecutive 
columns and and any $n$ rows. Figure~\ref{fig:H} represents an example of the graph $X_{n,n}$ with $n=6$.
\end{itemize}

\begin{figure}[ht]
\begin{center}
\unitlength=0.2mm
\begin{picture}(50,250)
\put(-100,0){\circle*{5}}
\put(-50,0){\circle*{5}}
\put(0,0){\circle*{5}}
\put(50,0){\circle*{5}}
\put(100,0){\circle*{5}}
\put(150,0){\circle*{5}}

\put(-100,50){\circle*{5}}
\put(-50,50){\circle*{5}}
\put(0,50){\circle*{5}}
\put(50,50){\circle*{5}}
\put(100,50){\circle*{5}}
\put(150,50){\circle*{5}}

\put(-100,100){\circle*{5}}
\put(-50,100){\circle*{5}}
\put(0,100){\circle*{5}}
\put(50,100){\circle*{5}}
\put(100,100){\circle*{5}}
\put(150,100){\circle*{5}}

\put(-100,150){\circle*{5}}
\put(-50,150){\circle*{5}}
\put(0,150){\circle*{5}}
\put(50,150){\circle*{5}}
\put(100,150){\circle*{5}}
\put(150,150){\circle*{5}}

\put(-100,200){\circle*{5}}
\put(-50,200){\circle*{5}}
\put(0,200){\circle*{5}}
\put(50,200){\circle*{5}}
\put(100,200){\circle*{5}}
\put(150,200){\circle*{5}}

\put(-100,250){\circle*{5}}
\put(-50,250){\circle*{5}}
\put(0,250){\circle*{5}}
\put(50,250){\circle*{5}}
\put(100,250){\circle*{5}}
\put(150,250){\circle*{5}}

\put(-50,0){\line(-1,0){50}}
\put(-50,0){\line(-1,1){50}}
\put(-50,0){\line(-1,2){50}}
\put(-50,0){\line(-1,3){50}}
\put(-50,0){\line(-1,4){50}}
\put(-50,0){\line(-1,5){50}}

\put(0,0){\line(-1,0){50}}
\put(0,0){\line(-1,1){50}}
\put(0,0){\line(-1,2){50}}
\put(0,0){\line(-1,3){50}}
\put(0,0){\line(-1,4){50}}
\put(0,0){\line(-1,5){50}}

\put(50,0){\line(-1,0){50}}
\put(50,0){\line(-1,1){50}}
\put(50,0){\line(-1,2){50}}
\put(50,0){\line(-1,3){50}}
\put(50,0){\line(-1,4){50}}
\put(50,0){\line(-1,5){50}}

\put(100,0){\line(-1,0){50}}
\put(100,0){\line(-1,1){50}}
\put(100,0){\line(-1,2){50}}
\put(100,0){\line(-1,3){50}}
\put(100,0){\line(-1,4){50}}
\put(100,0){\line(-1,5){50}}

\put(150,0){\line(-1,0){50}}
\put(150,0){\line(-1,1){50}}
\put(150,0){\line(-1,2){50}}
\put(150,0){\line(-1,3){50}}
\put(150,0){\line(-1,4){50}}
\put(150,0){\line(-1,5){50}}


\put(-50,50){\line(-1,0){50}}
\put(-50,50){\line(-1,1){50}}
\put(-50,50){\line(-1,2){50}}
\put(-50,50){\line(-1,3){50}}
\put(-50,50){\line(-1,4){50}}

\put(0,50){\line(-1,0){50}}
\put(0,50){\line(-1,1){50}}
\put(0,50){\line(-1,2){50}}
\put(0,50){\line(-1,3){50}}
\put(0,50){\line(-1,4){50}}

\put(50,50){\line(-1,0){50}}
\put(50,50){\line(-1,1){50}}
\put(50,50){\line(-1,2){50}}
\put(50,50){\line(-1,3){50}}
\put(50,50){\line(-1,4){50}}

\put(100,50){\line(-1,0){50}}
\put(100,50){\line(-1,1){50}}
\put(100,50){\line(-1,2){50}}
\put(100,50){\line(-1,3){50}}
\put(100,50){\line(-1,4){50}}

\put(150,50){\line(-1,0){50}}
\put(150,50){\line(-1,1){50}}
\put(150,50){\line(-1,2){50}}
\put(150,50){\line(-1,3){50}}
\put(150,50){\line(-1,4){50}}

\put(-50,100){\line(-1,0){50}}
\put(-50,100){\line(-1,1){50}}
\put(-50,100){\line(-1,2){50}}
\put(-50,100){\line(-1,3){50}}

\put(0,100){\line(-1,0){50}}
\put(0,100){\line(-1,1){50}}
\put(0,100){\line(-1,2){50}}
\put(0,100){\line(-1,3){50}}

\put(50,100){\line(-1,0){50}}
\put(50,100){\line(-1,1){50}}
\put(50,100){\line(-1,2){50}}
\put(50,100){\line(-1,3){50}}

\put(100,100){\line(-1,0){50}}
\put(100,100){\line(-1,1){50}}
\put(100,100){\line(-1,2){50}}
\put(100,100){\line(-1,3){50}}

\put(150,100){\line(-1,0){50}}
\put(150,100){\line(-1,1){50}}
\put(150,100){\line(-1,2){50}}
\put(150,100){\line(-1,3){50}}

\put(-50,150){\line(-1,0){50}}
\put(-50,150){\line(-1,1){50}}
\put(-50,150){\line(-1,2){50}}

\put(0,150){\line(-1,0){50}}
\put(0,150){\line(-1,1){50}}
\put(0,150){\line(-1,2){50}}

\put(50,150){\line(-1,0){50}}
\put(50,150){\line(-1,1){50}}
\put(50,150){\line(-1,2){50}}

\put(100,150){\line(-1,0){50}}
\put(100,150){\line(-1,1){50}}
\put(100,150){\line(-1,2){50}}

\put(150,150){\line(-1,0){50}}
\put(150,150){\line(-1,1){50}}
\put(150,150){\line(-1,2){50}}

\put(-50,200){\line(-1,0){50}}
\put(-50,200){\line(-1,1){50}}

\put(0,200){\line(-1,0){50}}
\put(0,200){\line(-1,1){50}}

\put(50,200){\line(-1,0){50}}
\put(50,200){\line(-1,1){50}}

\put(100,200){\line(-1,0){50}}
\put(100,200){\line(-1,1){50}}

\put(150,200){\line(-1,0){50}}
\put(150,200){\line(-1,1){50}}

\put(-50,250){\line(-1,0){50}}

\put(0,250){\line(-1,0){50}}

\put(50,250){\line(-1,0){50}}

\put(100,250){\line(-1,0){50}}

\put(150,250){\line(-1,0){50}}
\end{picture}
\end{center}
\caption{The graph $X_{6,6}$}
\label{fig:H}
\end{figure}
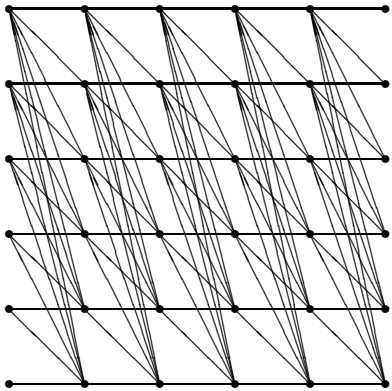

The unboundedness of clique-width in the class  ${\cal G}_{2^{\infty}}$ follows from the following result proved in \cite{cwbp}.

\begin{lemma}\label{lem:bp}
The clique-width of $X_{n,n}$ is at least $n/6$.
\end{lemma}

In what follows, we will prove that that every class ${\cal G}_{\alpha}$ with infinitely many $2$s in $\alpha$ has unbounded 
clique-width by showing that graphs in this class contain $X_{n,n}$ as a vertex minor for arbitrarily large values of $n$. 
We start with the case when the letter $1$ appears finitely many times in $\alpha$.

\begin{lemma}\label{ucw2}
Let $\alpha$ be an infinite word over the alphabet $\{0,1,2\}$, containing the letter $2$ infinitely many times and the letter $1$ finitely many times. 
Then the class ${\cal G}_{\alpha}$ has unbounded clique-width.
\end{lemma}

\begin{proof}
First fix a constant $n$. Let $\beta$ be a factor of $\alpha$ containing precisely $n$ instances of the letter $2$, 
starting and ending with the letter $2$ and containing no instances of the letter $1$ 
(since letter $2$ appears infinitely many times and letter $1$ finitely many times in $\alpha$, we can always find such a factor). 
We denote the length of $\beta$ by $\ell$ and consider the subgraph $G_n$ of $P^{\alpha}$ induced 
by $\ell+1$ consecutive columns corresponding to $\beta$ and by any $n2^{n-1}$ rows.  We will now show that $G_n$ contains the graph $X_{n,n}$ 
as a vertex-minor.

Using arguments identical to those in Theorem~\ref{ucw}, we can show that any instance of $00$ can be replaced by $0$ with the help of local complementations
and vertex deletions. 

Now each instance of $0$ is surrounded by $2$s in $\beta$. Consider any factor $02$ of $\beta$ 
and let $V_j, V_{j+1}, V_{j+2}$ be three columns such that $V_{j}\cup V_{j+1}$ induces a 1-regular graph and $V_{j+1}\cup V_{j+1}$ induces a chain graph. 
If we apply a local complementation to each vertex of $V_{j+1}$ in turn, it is easy to see that the edges between $V_{j}$ 
and $V_{j+2}$ form a chain graph. Looking at the vertices in the column $V_{j+2}$ we see that for any two vertices $x$ and $y$, 
when a local complementation is applied at $z\in V_{j+1}$ the edge between $x$ and $y$ is complemented if and only if both $x$ and $y$ are adjacent to $z$. 
Therefore, $x$ and $y$ are adjacent  if and only if $\min \{ |N(x)\cap V_{j+1}|, |N(y)\cap V_{j+1}|\}$ is odd. 
Hence the vertices of $V_{j+2}$ in the even rows induce an independent set. 
So, applying a local complementation to each vertex of $V_{j+1}$ in turn and then deleting column $V_{j+1}$ together with the odd rows 
allows us to reduce the factor $02$ to $2$. This transformation also reduces the number of rows two times. Since the factor $02$ can 
appear at most $n-1$ times, in at most $n-1$ transformations we reduced $G_n$ to a graph containing $X_{n,n}$. Therefore, $G_n$ contains $X_{n,n}$ as a vertex minor. 

Since $n$ can be arbitrarily large, we conclude with the help of Lemma~\ref{lem:bp} that graphs in ${\cal G}_{\alpha}$
can have arbitrarily large clique-width.
\end{proof}

\medskip
To extend the last lemma to a more general result, we again refer to \cite{oum2}, which introduces another useful transformation, called {\it pivoting}. 
For a graph $G$ and an edge $xy$, the graph obtained by pivoting $xy$ is defined to be the graph obtained by applying local complementation 
at $x$, then at $y$ and then at $x$ again. Oum shows in \cite{oum2} that in the case of bipartite graphs pivoting $xy$ is identical to complementing 
the edges between $N(x)\setminus \{ y\}$ and $N(y)\setminus \{ x\}$. We will use this transformation to prove the following result.

\begin{lemma}\label{ucw3}
Let $\alpha$ be an infinite word over the alphabet $\{0,1,2\}$, containing the letter $2$ infinitely many times. 
Then the class ${\cal G}_{\alpha}$ has unbounded clique-width.
\end{lemma}

\begin{proof}
First, fix a constant $n$. Let $\beta$ be a factor of $\alpha$ containing precisely $n$ instances of the letter $2$, 
starting and ending with the letter $2$. Let $G_n$ be the subgraph of ${\cal P}^{\alpha}$ induced by the columns 
corresponding to $\beta$ and by any $n2^n+n^2$ rows. To prove the lemma, it is enough to show that $G_n$ contains either $F_{n,n}$ or
$X_{n,n}$ as a vertex minor.

Consider any two consecutive appearances of $2$ in $\beta$ and denote 
the word between them by $\gamma$. In other words, $\gamma$ is a (possibly empty) word in the alphabet $\{0,1\}$.
If $\gamma$ contains at least $n$ instances of $1$, then by Lemma~\ref{ucw} $G_n$ contains $F_{n,n}$ as a vertex minor.
Therefore, we assume that the number of $1$s in $\gamma$ is at most $n-1$. If $\gamma$ contains no instance of $1$, then we apply 
the idea of Lemma~\ref{ucw2} to reduce it to the empty word. If $\gamma$ contains at least one instance of $1$, we apply the idea of 
Lemma~\ref{ucw} to eliminate all $0$s from it. 

Suppose that after this transformation $\gamma$ contains at least two $1$s, that is,  $\beta$ contains $211$ as a factor.
Let $V_j, V_{j+1}, V_{j+2}$ and $V_{j+3}$ be the four columns such that $V_{j+1}\cup V_{j+2}$ and 
$V_{j+2}\cup V_{j+3}$ induce bipartite complements of 1-regular graph and $V_j\cup V_{j+1}$ induces a chain graph.
Let $x$ be the vertex in the first row of column $V_{j+1}$ and $y$ be the vertex in the last row of column $V_{j+2}$.
It is not difficult to see that if we pivot the edge $xy$ and delete the first and the last row, then the graphs induced by $V_{j+1}\cup V_{j+2}$
and by $V_{j+2}\cup V_{j+3}$ become a 1-regular. In other words, we transform the factor $211$ into $200$. 
Then we apply the idea of Lemma~\ref{ucw} to further transform it into $2$. 

Repeated applications of the above transformation allows us to assume that $\gamma$ contains exactly one
$1$, that is,  $\beta$ contains $212$ as a factor.
Let $V_j, V_{j+1}, V_{j+2}$ and $V_{j+3}$ be the four columns such that 
$V_j\cup V_{j+1}$ and $V_{j+2}\cup V_{j+3}$ induce chain graphs and $V_{j+1}\cup V_{j+2}$ induces the bipartite complement of a 1-regular graph. 
Let $x$ be the vertex in the first row of column $V_{j+1}$ and $y$ be the vertex in the last row of column $V_{j+2}$.
It is not difficult to see that if we pivot the edge $xy$ and delete the first and the last row, then the graph induced by $V_{j+1}\cup V_{j+2}$
becomes 1-regular, while the graphs induced by $V_{j}\cup V_{j+1}$ and by $V_{j+2}\cup V_{j+3}$ remain chain graphs. 
In other words, we transform the factor $212$ into $202$. Then we apply the idea of Lemma~\ref{ucw2} to further transform it into $22$.

The above procedure applied at most $n-1$ times allows us to transform $\beta$ into the word of $n$ consecutive $2$s. In terms of graphs, $G_n$ transforms into a sequence 
of $n$ chain graphs. Moreover, it is not difficult to see that if initially $G_n$ contains $n2^n+n^2$ rows, then the resulting graph has 
at least $n$ rows, that is, it contains $X_{n,n}$ as a vertex minor.   
\end{proof}


\section{Conclusion and open problems}
\label{sec:conclusion}

In the preceding sections, we have described a new family of hereditary classes of graphs of unbounded clique-width. 
For many of them, we proved the minimality.  Our results allow us to make the following conclusion.

\begin{theorem}\label{thm:main}
There exist infinitely many minimal hereditary classes of graphs of unbounded clique-width and linear clique-width. 
\end{theorem}

\begin{proof}
Let $n$ be a natural number and $\alpha^{(n)}=(0^n1)^{\infty}$. Since $\alpha^{(n)}$ is an infinite periodic word,
by Theorem~\ref{thm:min} ${\cal G}_{\alpha^{(n)}}$ is a minimal class of unbounded clique-width and linear clique-width.

If $n<m$, then ${\cal G}_{\alpha^{(n)}}$ and ${\cal G}_{\alpha^{(m)}}$ do not coincide, since 
${\cal G}_{\alpha^{(n)}}$ contains an induced $C_{2(n+2)}$, while ${\cal G}_{\alpha^{(m)}}$ does not (which it is not difficult to see).
Therefore,  ${\cal G}_{\alpha^{(1)}}$, ${\cal G}_{\alpha^{(2)}},\ldots$ is an infinite sequence of  
minimal hereditary classes of graphs of unbounded clique-width and linear clique-width.
\end{proof}

\medskip
A full description of minimal classes of the form ${\cal G}_{\alpha}$ remains an open question. 
To propose a conjecture addressing this question, we first define the notion of 
almost periodic word.
An infinite word $\alpha$ is {\it almost periodic} if for each factor $\beta$ of $\alpha$ there exists a constant $\ell(\beta)$ 
such that every factor of $\alpha$ of length at least $\ell(\beta)$ contains $\beta$ as a factor.

\begin{conjecture}\label{conj:min_classes}
  Let $\alpha$ be an infinite word over the alphabet $\{0,1,2\}$.
  Then the class $\mathcal{G}_\alpha$ is a minimal hereditary class of unbounded clique-width
  if and only if $\alpha$~is almost periodic and contains at least one $1$ or~$2$.
\end{conjecture} 

Note that almost periodicity implies that either $1$ or $2$
appears in $\alpha$ infinitely many times. It is not hard to verify that this condition is
necessary for the class ${\cal G}_{\alpha}$ to have unbounded clique-width. In other words,
if $\alpha$ contains finitely many $1$s and $2$s the class ${\cal G}_{\alpha}$ has bounded
clique-width.

We conclude the paper by discussing an intriguing relationship between clique-width in a hereditary class $X$ and the existence of infinite antichains 
in $X$ with respect to the induced subgraph relation. 
In particular, the following question was asked in \cite{DRT10}: 
is it true that if the clique-width in $X$ is unbounded, then it necessarily contains an infinite antichain?
Recently, this question was answered negatively in \cite{LRZ15}. However, in the case of so-called {\it coloured}
induced subgraphs, the question remains open.  

\begin{itemize}
\item[] {\it Coloured induced subgraphs}. 
We define this notion for two colours, which is the simplest case where the above question is open.
Assume we deal with graphs whose vertices are coloured by two colours, say white and black. 
We say that a graph $H$ is a {\it coloured induced subgraph} of $G$ if there is an embedding 
of $H$ into $G$ that respects the colours. With this strengthening of the induced subgraph relation, 
some graphs that are comparable without colours may become incomparable if equipped with colours. 
Consider, for instance, two chordless paths $P_k$ and $P_n$. Without colours, one of them is an induced subgraph of the other. 
Now imagine that we colour the endpoints of both paths black and the remaining vertices white. Then clearly they become incomparable with
respect to the coloured induced subgraph relation (if $k\ne n$). Therefore, the set of all paths coloured in this way create an infinite coloured antichain.
Let us denote it by $A_0$. 
\end{itemize}

In \cite{DRT10}, it was conjectured that hereditary classes of graphs of unbounded clique-width necessarily contain infinite coloured antichains.
We believe this is true. Moreover, we propose the following strengthening of the conjecture from \cite{DRT10}.

\begin{conjecture}\label{con:1}
Every minimal hereditary class of graphs of unbounded clique-width contains a canonical infinite coloured antichain.
\end{conjecture}

The notion of a canonical antichain was introduced by Guoli Ding in \cite{canonical} and can be defined for
hereditary classes as follows. An infinite antichain $A$ in a hereditary class $X$ is \emph{canonical} if any hereditary
subclass of $X$ containing only finitely many graphs from $A$ has no infinite antichains. In other words,
speaking informally, an antichain is canonical if it is unique in the class.

\medskip
To support Conjecture~\ref{con:1}, let us observe that it is valid for all minimal classes ${\cal G}_{\alpha}$ described in Theorem~\ref{thm:min}.
Indeed, all of them contain arbitrarily large chordless paths and hence all of them contained the infinite coloured antichain $A_0$ defined above.
Moreover, this antichain is canonical, because by forbidding all paths of length greater than $k$ for some fixed $k$, we are left 
with subgraphs of $P^{\alpha}$ occupying at most $k$ consecutive columns, in which case the clique-width of such graphs is at most $4k$
by Lemma~\ref{easy}.

There exist many other infinite coloured antichains, but all available examples are obtained from the antichain $A_0$ by various transformations.
We believe that any infinite coloured antichain can be transformed from $A_0$ in a certain way and that any minimal hereditary class of unbounded 
clique-width can be transformed from ${\cal P}^{\alpha}$ (for some $\alpha$) in a similar way. Describing the set of these transformations is a challenging research task.


\begin{thebibliography}{99}


\bibitem{cwbp}
A. Brandst{\"a}dt and V.V. Lozin, 
On the linear structure and clique-width of bipartite permutation graphs,
{\it Ars Combinatoria}, 67 (2003) 273--281.


\bibitem{Robert}
R. Brignall, N. Korpelainen, V. Vatter,
Linear Clique-Width for Hereditary Classes of Cographs,
{\it J. Graph Theory}, accepted (\verb|DOI: 10.1002/jgt.22037|).

\bibitem{CER93}
B.~Courcelle, J.~Engelfriet and G.~Rozenberg,
Handle-rewriting hypergraph grammars, \emph{Journal of Computer and
System Sciences} 46 (1993) 218--270.


\bibitem{CMR00}
B.~Courcelle, J.A.~Makowsky and U.~Rotics,
Linear time solvable optimization problems on graphs of bounded
clique-width, {\em Theory Comput.~ Syst.} 33 (2000) 125--150.

\bibitem{CO00}
B.~Courcelle and S.~Olariu,
 Upper bounds to the clique-width of a graph,
{\em Discrete Applied Math.} 101 (2000) 77--114.

\bibitem{DRT10}
J. Daligault, M. Rao and S. Thomass\'e, 
Well-Quasi-Order of Relabel Functions, {\it Order} 27 (2010) 301--315.

\bibitem{canonical}
G. Ding, On canonical antichains, 
{\it Discrete Mathematics} 309 (2009) 1123-–1134.

\bibitem{Lozin}
V.V. Lozin, Minimal classes of graphs of unbounded clique-width,
{\it Annals of Combinatorics}, 15 (2011) 707--722.

\bibitem{LRZ15}
V. Lozin, I. Razgon, V. Zamaraev,
Well-quasi-ordering does not imply bounded clique-width,
{\it Lecture Notes in Computer Science}, 9224 (2016) 351--359.



\bibitem{oum}
S. Oum, P. Seymour, Approximating clique-width and branch-width, {\it Journal of Combinatorial Theory, Series B}, 96 (4) (2006) 514-528

\bibitem{oum2}
S. Oum, Rank-width and vertex-minors, {\it Journal of Combinatorial Theory, Series B}, 95 (1) (2005) 79-100


\bibitem{RS86}
N. Robertson and P.D. Seymour, Graph minors. V. Excluding a planar graph, 
{\it J. Combinatorial Theory Ser. B}, 41 (1986) 92--114. 

\bibitem{RS88}
N. Robertson and P.D. Seymour, 
Graph Minors. XX. Wagner's conjecture, {\it Journal of Combinatorial Theory Ser. B}, 92 (2004) 325--357. 

\end{thebibliography}
\end{document}